\documentclass[oneside]{amsart}
\usepackage[pdfdisplaydoctitle=true,
            colorlinks=true,
            urlcolor=blue,
            citecolor=blue,
            linkcolor=blue,
            pdfstartview=FitH,
            pdfpagemode= UseNone,
            bookmarksnumbered=true]{hyperref}
\usepackage{amsmath}
\usepackage{amsfonts}
\usepackage{amssymb}
\usepackage[dvipsnames]{xcolor}
\usepackage{mathrsfs}
\usepackage{enumerate}
\usepackage{amsthm}

\makeatletter
\def\@tocline#1#2#3#4#5#6#7{\relax
  \ifnum #1>\c@tocdepth 
  \else
    \par \addpenalty\@secpenalty\addvspace{#2}%
    \begingroup \hyphenpenalty\@M
    \@ifempty{#4}{%
      \@tempdima\csname r@tocindent\number#1\endcsname\relax
    }{%
      \@tempdima#4\relax
    }%
    \parindent\z@ \leftskip#3\relax \advance\leftskip\@tempdima\relax
    \rightskip\@pnumwidth plus4em \parfillskip-\@pnumwidth
    #5\leavevmode\hskip-\@tempdima
      \ifcase #1
       \or\or \hskip 1em \or \hskip 2em \else \hskip 3em \fi%
      #6\nobreak\relax
    \hfill\hbox to\@pnumwidth{\@tocpagenum{#7}}\par
    \nobreak
    \endgroup
  \fi}
\makeatother

\newtheorem{theorem}{Theorem}[section]
\newtheorem{lemma}[theorem]{Lemma}
\newtheorem{proposition}[theorem]{Proposition}
\newtheorem{corollary}[theorem]{Corollary}

\theoremstyle{definition}

\newtheorem{example}[theorem]{Example}

\newtheorem*{theorem*}{Theorem}
\theoremstyle{remark}

\numberwithin{equation}{section}

\newcommand{\norm}[1]{\left\lVert#1\right\rVert}
\newcommand{\bignorm}[1]{\big\lVert#1\big\rVert}
\newcommand{\Bignorm}[1]{\Big\lVert#1\Big\rVert}

\newcommand{\ip}[1]{\left\langle #1 \right\rangle}
\newcommand{\bigip}[1]{\big\langle #1 \big\rangle}
\newcommand{\Bigip}[1]{\Big\langle #1 \Big\rangle}

\newcommand{\R}{\mathbb{R}}

\newcommand{\D}{\mathscr{D}}

\newcommand{\Dex}{\mathscr{D}_{\mu,\text{ex}}}

\newcommand{\n}{\noindent}

\newcommand{\symp}{\nabla^\perp}

\DeclareMathOperator{\Ad}{Ad}
\DeclareMathOperator{\ad}{ad}

\author[1]{Martin Bauer}
\author[2]{Patrick Heslin}
\author[3]{Gerard Misio{\l}ek}
\author[4]{Stephen C. Preston}
\thanks{M. Bauer was partially supported
by NSF grant DMS-1953244 and by FWF grant FWF-P 35813-N. P. Heslin was supported by the National University of Ireland's Dr. {\'E}amon de Valera Postdoctoral Fellowship}
\address{M. Bauer: Florida State University, USA}
\email{bauer@math.fsu.edu}
\address{P. Heslin: National University of Ireland, Maynooth, Ireland}
\email{patrick.heslin@mu.ie}
\address{G. Misio{\l}ek: University of Notre Dame, USA}
\email{gmisiole@nd.edu}
\address{S.C. Preston: Brooklyn College and CUNY Graduate Center, USA}
\email{stephen.preston@brooklyn.cuny.edu}

\begin{document}

\title[Geometric Analysis of the Generalized SQG Equations]{Geometric Analysis of the Generalized Surface Quasi-Geostrophic Equations}




\begin{abstract}
We investigate the geometry of a family of equations in two dimensions which interpolate between the Euler equations of ideal hydrodynamics and the inviscid surface quasi-geostrophic equation. This family can be realised as geodesic equations on groups of diffeomorphisms. We show precisely when the corresponding Riemannian exponential map is non-linear Fredholm of index 0. We further illustrate this by examining the distribution of conjugate points in these settings via a Morse theoretic approach.
\end{abstract}

\maketitle

\tableofcontents
\section{Introduction}
In the early 1990s Constantin et al. \cite{constantin1994formation} investigated a 2D hydrodynamic model known as the surface quasi-geostrophic (SQG) equation. One motivation for this was a striking mathematical analogy with the 3D Euler equations exhibited by the authors via several key analytic similarities. These investigations led to the study of a whole family of inviscid SQG-type equations which can be viewed as interpolating between the standard SQG equation and the 2D Euler equations, cf. \cite{chae2011inviscid}.

On the other hand, it is well known from the work of Arnold and his school that the motions of an ideal fluid have a beautiful geometric description as geodesic equations on the group of volume-preserving diffeomorphisms of the fluid domain \cite{arnold1966sur, arnold2021topological}. A key distinction between two and three dimensional hydrodynamics comes from examining the Fredholm properties of the associated Riemannian exponential maps. In 2D the exponential map is a non-linear Fredholm map of index zero, but in 3D it is not \cite{ebin2006singularities, misiolek2010fredholm, shnirelman2005microglobal}.

It was shown by Washabaugh in \cite{washabaugh2016the} that the standard SQG equation can also be reformulated as a geodesic equation on the group of exact volume-preserving diffeomorphisms equipped with a metric given by the homogeneous Sobolev $\dot{H}^{-\frac{1}{2}}$ inner product. He further illustrated via an explicit example that the exponential map in this setting fails to be Fredholm. Recently one of the authors and Vu showed that the entire generalized SQG family, given by
\begin{equation}\label{beta SQG}
    \begin{split}
        &\partial_t \theta + u \cdot \nabla \theta = 0 \\
        &u=\symp (-\Delta)^{\frac{\beta}{2}-1}\theta, \qquad 0\leq \beta \leq 1
    \end{split}
\end{equation}
can be viewed as geodesic equations \cite{misiolek2023on}. This opens up a possibility to investigate, using Riemannian geometric tools, the transition from 2D Euler to the surface quasi-geostrophic equations.

The central focus of this paper is to understand how the parameter $\beta$ affects properties of the corresponding Riemannian exponential maps. In particular, we prove that for $\beta < 1$ these exponential maps are non-linear Fredholm of index $0$, illustrating that the value $\beta=1$ is critical. In the former case we establish, via a Morse theoretic argument, a finite upper bound on the number of conjugate points along an arbitrary geodesic segment. As the interpolation parameter approaches the critical value, this upper bound approaches infinity. We then provide an explicit example to show that this bound is attained.

\section{Preliminaries}\label{preliminaries}
We will aim to keep this section reasonably general and concise. Full details of all the constructions can be found in \cite{arnold2021topological, ebin1970groups, misiolek2010fredholm} and the references therein.
\subsection{Configuration Spaces}
Let $M$ be a compact surface without boundary equipped with a Riemannian metric $g$ with volume form $\mu$. Let $\nabla$, $\symp$, $\mathrm{div}$ and $\Delta$ denote the associated gradient, symplectic gradient, divergence and Laplacian operators. If $E\rightarrow M$ is a smooth vector bundle over $M$, we let $H^s(E)$ denote those sections of $E$ with $L^2$ derivatives of order $s$. In particular, $H^s(TM)$ is the space of $H^s$ vector fields on $M$ which we equip with an inner product via
\begin{equation*}
    \ip{u,v}_{H^s} = \ip{u,(1 + \Delta)^s v}_{L^2}.
\end{equation*} We define the set $H^s(M,M)$ to consist of all functions from $M$ to itself which are of class $H^s$ in every chart. By the Sobolev Embedding Theorem we have that, for $s > k+1$, the space $H^s(M,M)$ continuously embeds into $C^k(M,M)$. If $s>1$, then it follows that $H^s(M,M)$ is a smooth Hilbert manifold. If $s > 2$, then the set of $H^s$-diffeomorphisms $\D^s(M)=\{C^{1}\text{ diffeomorphisms of } M\}\cap H^s(M,M)$ inherits a smooth submanifold structure as an open subset of $H^s(M,M)$. Its tangent space at the identity,  $T_e\D^s(M)=H^s(TM)$, has an $L^2$-orthogonal decomposition by the Hodge theorem:
\begin{equation}\label{hodge}
    H^s(TM) = \mathcal{H} \oplus \nabla H^{s+1}(M) \oplus \symp H^{s+1}(M)
\end{equation}
where $\mathcal{H}$ denotes the finite-dimensional subspace of harmonic vector fields on $M$.

The subgroup $\Dex^s(M)$ of exact diffeomorphisms is then defined to be the submanifold whose tangent space at the identity consists of those divergence-free vector fields with zero harmonic component, i.e. 
\begin{equation}
T_e\Dex^s(M) = \symp H^{s+1}(M) = \{u \ \vert \ u=\symp \psi_u \ \text{with} \ \psi_u \in H^{s+1 }(M)\}.
\end{equation}
We refer to $\psi_u$ as the \textit{stream function} for $u$. This configuration space will be the central focus of the paper, however we will also consider its smooth counterpart $\Dex(M) = \{C^{\infty} \text{ diffeomorphisms of } M\} \cap \Dex^s(M)$ where $T_e\Dex(M) = \symp C^{\infty}(M)$. 

It is worth noting at this point that the homogeneous Sobolev $\dot{H}^s$ inner product, given by \[\displaystyle \ip{u,v}_{\dot{H}^s} = \ip{u,\Delta^s v}_{L^2},\] when restricted to $T_e\Dex^s(M)$, induces a topology equivalent to the standard $H^s$ topology.

$\Dex^s(M)$ is further a topological group under composition where right translation $R_\eta$ is smooth, but left translation $L_\eta$ is just continuous, in the $H^s$ topology. The group adjoint operator is given by:
\begin{align}\label{exact group adjoint}
    \Ad_\eta v &= d_{\eta^{-1}}L_\eta d_eR_{\eta^{-1}} v = \big(D\eta \cdot v \big) \circ \eta^{-1} \\ \nonumber
    &= \big(D\eta \cdot \symp \psi_v\big) \circ \eta^{-1} = \symp \big(\psi_v \circ \eta^{-1} \big)
\end{align}
and the Lie algebra adjoint\footnote{ \ Strictly speaking, these groups of Sobolev diffeomorphisms are not Lie groups. For example, if $u,v \in T_e\Dex^s(M)$, their commutator is a priori only of Sobolev class $H^{s-1}$. However, the group structure they possess will be sufficient for our purposes.} by:
\begin{equation}\label{exact lie algebra adjoint}
    \ad_u v = -[u,v] = \symp \{\psi_u,\psi_v\}
\end{equation}
where $\{\psi_u,\psi_v\}=g(\symp\psi_v, \nabla \psi_u)$ is a Poisson bracket.

\subsection{Right-Invariant Metrics and their Exponential Maps}
We now recall the abstract geometric framework of V. Arnold \cite{arnold1966sur, arnold2021topological}.

Let $\ip{\cdot, \cdot}$ be an inner product on $T_e\Dex^s(M)$ which we extend to a Riemannian metric on $\Dex^s(M)$ via right translation. For a linear operator $L: T_e\Dex^s(M) \rightarrow T_e\Dex^s(M)$, we denote by $L^*$ its adjoint with respect to the metric, i.e, $\ip{Lu,v} = \ip{u, L^*v}$ for all $u,v \in T_e\Dex^s(M)$. As their definitions involve both the Lie and Riemannian structure, the geometry is in some sense encoded in the operators
$$\ip{\Ad^*_\eta u , v} = \ip{u, \Ad_\eta v} \quad \text{and} \quad \ip{\ad_u^*v, w} = \ip{v, \ad_u w}$$
associated with the Lie group and Lie algebra coadjoint representations. Geodesics $t \mapsto \gamma(t)$ in $\Dex^s(M)$ of the metric induced by $\ip{\cdot, \cdot}$ are critical paths for the energy functional. The corresponding geodesic equation, when reduced to the Lie algebra, is referred to as an \textit{Euler-Arnold} equation which takes the form:
\begin{align}\label{euler-arnold}
\begin{split}
    &\partial_t u = - \ad^{*}_{u}u \\
    & u = \dot{\gamma} \circ \gamma^{-1}.
\end{split}
\end{align}
If we further impose an initial condition $u(0)=u_0$, then we obtain an initial value problem which enjoys the following (coadjoint) conservation law
\begin{equation}\label{Lie group conservation law}
    \Ad^{{*}}_{\gamma(t)}u(t) = u_0.
\end{equation}

For a comprehensive list of notable examples of Euler-Arnold equations, cf. \cite{arnold2021topological, misiolek2010fredholm, vizman2008geodesic}. The second equation in \eqref{euler-arnold} is simply the flow equation, which provides a link between the (Eulerian) solution map $u_0 \mapsto u$ and the (Lagrangian) solution map $u_0 \mapsto (\gamma, \dot{\gamma})$. The latter has a clear geometric interpretation as a Riemannian exponential map on $\Dex^s(M)$. Namely, we have
\begin{equation}\label{lagrangian solution map}
    u_0 \mapsto \exp_e(tu_0) = \gamma(t)
\end{equation}
where $\gamma(t)$ is the unique geodesic with $\gamma(0)=e$ and $\dot{\gamma}(0)=u_0$.

\subsection{The Jacobi Equation}
Assume now that the metric induced by $\ip{\cdot, \cdot}$ has a well defined smooth exponential map on $\Dex^s(M)$. For fixed $u_0 \in T_e\Dex^s(M)$ and $T>0$, we have the geodesic segment, $\gamma \vert_{[0,T]}$, where $\gamma(t)=\exp_e(tu_0)$. The Jacobi field along $\gamma$ with initial conditions $J(0)=0$ and $\dot{J} (0) = w_0$ is then given by $J(t) = d\exp_e(tu_0)t w_0$. If we left-translate back to the Lie algebra via $v_L(t) = d_eL_{\eta^{-1}(t)}J(t)$, the field $v_L$ satisfies the system
\begin{equation}\label{left jacobi field}
\begin{split}
    &\partial_t v_L = w \\
    & \partial_t(\Ad^*_{\gamma}\Ad_{\gamma}w) + \ad^*_{w}(u_0) = 0
\end{split}
\end{equation}
with initial data $v_L(0) = 0$, $\partial_t v_L (0) = w_0$. We denote the solution operator of this Cauchy problem by
\begin{equation}
    \Phi(t): T_e\Dex^s \rightarrow T_e\Dex^s, \qquad \Phi(t) w_0 = v_L(t)
\end{equation}
and recall the following result from \cite{misiolek2010fredholm}.
\begin{lemma}\label{decomp}
    Let $u_0\in T_e\Dex$ with $\Lambda(t)$ and $K_{0}$ the operators defined by the formulas
    \begin{equation}\label{lambda and K}
        w \mapsto \Lambda(t)w=\Ad^*_{\gamma(t)}\Ad_{\gamma(t)}w \qquad \text{ and } \qquad w \mapsto K_{0}(w) = \ad^*_w u_0.
    \end{equation}
    Then the solution operator $w \mapsto \Phi(t)w= t d_eL_{\gamma(t)}^{-1} d\exp_e (tu_0)w$ can be decomposed as
    \begin{equation}\label{phi}
        \Phi(t) = \Omega(t) + \Gamma(t)
    \end{equation}
    where
    \begin{equation}\label{omega and gamma}
        \Omega(t) = \int_0^t \Lambda(\tau)^{-1} \ d\tau \qquad \text{ and } \qquad \Gamma(t) = \int_0^t \Lambda(\tau)^{-1}K_{0}\Phi(\tau) \ d\tau.
    \end{equation}
\end{lemma}

\section{The SQG Family of Equations}

\subsection{The Euler-Arnold Equations and Well-Posedness}
\par We now equip the space of $H^s$-Sobolev exact volume-preserving diffeomorphisms of a closed two-dimensional Riemannian manifold $M$ with a family of right-invariant metrics\footnote{ \ More precisely, these are \textit{weak} Riemannian metrics on $\Dex^s(M)$ in the sense that they induce a weaker topology than the inherent $H^s$ topology.} given at the identity map by:
\begin{equation}\label{beta metric}
\begin{split}
    \ip{u,v}_{\dot{H}^{-\beta/2}} &= \bigip{\symp \psi_u,(-\Delta)^{-\frac{\beta}{2}}\symp \psi_v}_{L^2} \\
    &= \int_{M} \psi_u (-\Delta)^{1-\frac{\beta}{2}} \psi_v \ d\mu \ , \qquad 0\leq \beta \leq 1
\end{split}
\end{equation}
where again $\psi_u$ and $\psi_v$ are the stream functions for $u$ and $v$. Using \eqref{exact group adjoint} and \eqref{exact lie algebra adjoint} we compute the group and Lie algebra coadjoints
\begin{align*}
    \ip{\Ad^{*}_\eta u , v}_{\dot{H}^{-\beta/2}} &= \ip{u, \Ad_\eta v}_{\dot{H}^{-\beta/2}} \\
    &= \ip{\symp \psi_u , \symp (\psi_v \circ \eta^{-1})}_{\dot{H}^{-\beta/2}} \\
    &= \int_{M} (-\Delta)^{1-\frac{\beta}{2}} \psi_u \cdot \psi_v \circ \eta^{-1} \ d\mu
\end{align*}
and
\begin{align*}
    \ip{\ad_u^{*}v, w}_{\dot{H}^{-\beta/2}} &= \ip{v, \ad_u w }_{\dot{H}^{-\beta/2}} \\
    &= \ip{\symp \psi_v, \symp \{\psi_u, \psi_w\}}_{\dot{H}^{-\beta/2}} \\
    &= \int_{M} (-\Delta)^{1-\frac{\beta}{2}}\psi_v \{\psi_u, \psi_w\} \ d\mu \\
    &= \int_{M} \{\psi_u, (-\Delta)^{1-\frac{\beta}{2}}\psi_v\} \psi_w \ d\mu
\end{align*}
giving us
\begin{equation}\label{group coadjoint}
    \Ad^{*}_\eta u = \symp (-\Delta)^{\frac{\beta}{2}-1} R_\eta (-\Delta)^{1-\frac{\beta}{2}} \psi_u,
\end{equation}
and
\begin{equation}\label{lie algebra coadjoint}
    \ad_u^{*}v = \symp (-\Delta)^{\frac{\beta}{2}-1}\{(-\Delta)^{1-\frac{\beta}{2}}\psi_v, \psi_u\}.
\end{equation}

Hence, for the metrics \eqref{beta metric} the Euler-Arnold equations \eqref{euler-arnold} become the generalized SQG equations \eqref{beta SQG}
\begin{equation*}
    \begin{split}
        &\partial_t \theta + u \cdot \nabla \theta = 0 \\
        &u=\symp (-\Delta)^{\frac{\beta}{2}-1}\theta, \qquad  0\leq \beta \leq 1.
    \end{split}
\end{equation*}
Each of the $\beta$-metrics in \eqref{beta metric} has an associated Riemannian exponential map as in \eqref{lagrangian solution map}. The following proposition is a direct consequence of the results in \cite{misiolek2023non}.
\begin{proposition}\label{truong}
    For $s>3$ and $0\leq\beta\leq1$ the Riemannian exponential map on $\Dex^s(M)$ of the metric \eqref{beta metric} is $C^\infty$ and, consequently, the differential \[d\exp_e(tu_0): T_e\Dex^s(M) \rightarrow T_{\exp_e(tu_0)}\Dex^s \ , \qquad u_0 \in T_e\Dex^s(M)\] is a bounded linear operator.
\end{proposition}

In particular, since a standard calculation shows that $d \exp_e(0)$ is the identity, it follows from Proposition \ref{truong} that the exponential map is a local diffeomorphism near the identity map in $\Dex^s(M)$.

A further consequence of this is local well-posedness in the sense of Hadamard of the generalized SQG family \eqref{beta SQG}. To the best of the authors' knowledge, for $0<\beta\leq1$, the situation for global well-posedness is not known. However, a sufficient criterion for global regularity is obtained in \cite{chae2011inviscid}. Our aim now is to study how the parameter $\beta$ affects the analytic and geometric properties of the corresponding exponential maps.

\subsection{Fredholm Properties of the Exponential Maps}\label{beta fredholmness}
Recall that a bounded linear operator $F$ between Banach spaces is Fredholm if its range is closed and its kernel and cokernel are finite-dimensional. The difference of the dimensions of these two spaces is a topological invariant known as the index of $F$. The operator is semi-Fredholm if either of the finiteness requirements on its kernel or cokernel is dropped. Throughout the proceeding arguments, we will also use $\norm{\cdot}_{X}$ to denote the operator norm on the space of bounded linear operators on $X$, which should be clear from the context.

A map $f$ between Banach manifolds is said to be Fredholm if its differential is a Fredholm operator on corresponding tangent spaces. If the source manifold is connected then we define the index of $f$ to be the index of its differential $df$. Standard references for the above are \cite{kato1966pertubation, smale1965an}.

The goal of this section is to establish the following theorem, which extends to fractional orders the results of \cite{misiolek2010fredholm}.

\begin{theorem}\label{fredholmness}
    For $s>3$ and $0\leq \beta<1$, the exponential map on the group of exact volume-preserving diffeomorphisms, $\Dex^s$, induced by the metric \eqref{beta metric} is a non-linear Fredholm map of index 0.
\end{theorem}

The value $\beta=1$ is critical in the sense that the exponential map fails to be Fredholm. We will return to this in the following section.

\begin{proof}[Proof of Theorem \ref{fredholmness}]
    We begin by taking initial data $u_0 \in T_e\Dex^{s'}$ for some $s'\gg s$ and letting $\gamma(t)=\exp_e(tu_0)$ denote the corresponding geodesic in $\Dex^{s'}$, which lives for some time $T>0$. From Lemma \ref{decomp}, for any $w \in T_e\Dex^s$ and $t\in[0,T]$, we have the decomposition
    \begin{equation*}
        d\exp_e(tu_0)tw = D\gamma(t) \left(\Omega(t)w + \Gamma(t)w\right)
    \end{equation*}
    where $\Omega(t)$ and $\Gamma(t)$ are given by \eqref{omega and gamma}. Hence we estimate:
    \begin{align}\label{estimate 1}
        \begin{split}
            \norm{d\exp_e(tu_0)tw}_{\dot{H}^s} &\gtrsim \norm{\Omega(t)w}_{\dot{H}^s} - \norm{\Gamma(t)w}_{\dot{H}^s}
        \end{split}
    \end{align}
    where $\gtrsim$ refers to some constant depending on the uniform (in time) $\dot{H}^s$-norm of $D\gamma$ which we suppress for convenience. Focusing on the first term on the right side of \eqref{estimate 1} we have
    \begin{align*}
        \ip{w,\Omega(t)w}_{\dot{H}^s} &= \int_0^t \bigip{w, \Lambda(\tau)^{-1}w}_{\dot{H}^s} \ d\tau \\
        &= \int_0^t \Bigip{(-\Delta)^{\frac{2s+\beta}{4}}w, (-\Delta)^{\frac{2s+\beta}{4}}\Lambda(\tau)^{-1}w}_{\dot{H}^{-\beta/2}} \ d\tau \\
        &= \int_0^t \Bigip{(-\Delta)^{\frac{2s+\beta}{4}}w, \Lambda(\tau)^{-1}(-\Delta)^{\frac{2s+\beta}{4}}w}_{\dot{H}^{-\beta/2}} \ d\tau \\ 
        &\quad + \int_0^t \Bigip{(-\Delta)^{\frac{2s+\beta}{4}}w, [(-\Delta)^{\frac{2s+\beta}{4}}, \Lambda(\tau)^{-1}]w}_{\dot{H}^{-\beta/2}} \ d\tau.
    \end{align*}
From \eqref{exact group adjoint}, \eqref{group coadjoint} and \eqref{lambda and K} we acquire, for any $v \in T_e\Dex^s$
\begin{align}
    \Lambda(\tau) v &= \Ad_{\gamma(\tau)}^{*}\Ad_{\gamma(\tau)}v \\
    \nonumber&= \symp (-\Delta)^{\frac{\beta}{2}-1}R_{\gamma(\tau)}(-\Delta)^{1-\frac{\beta}{2}}R_{\gamma(\tau)}^{-1} \psi_v
\end{align}
which in turn yields
\begin{align}
    \Lambda(\tau)^{-1}v &= \Ad_{\gamma^{-1}(\tau)}\Ad_{\gamma^{-1}(\tau)}^{*}v \\
    \nonumber&= \symp R_{\gamma(\tau)} (-\Delta)^{\frac{\beta}{2}-1}R_{\gamma(\tau)}^{-1}(-\Delta)^{1-\frac{\beta}{2}} \psi_v.
\end{align}
Substituting the above and applying the Schwartz inequality, we have
\begin{align}
    \nonumber \ip{w,\Omega(t)w}_{\dot{H}^s} &= \int_0^t \Bigip{(-\Delta)^{\frac{2s+\beta}{4}}w, \Ad_{\gamma^{-1}}\Ad_{\gamma^{-1}}^{*}(-\Delta)^{\frac{2s+\beta}{4}}w}_{\dot{H}^{-\beta/2}} \ d\tau \\
    \nonumber &\quad+ \int_0^t \Bigip{(-\Delta)^{\frac{2s+\beta}{4}}w, \symp \Big[(-\Delta)^{\frac{2s+\beta}{4}}, R_{\gamma} (-\Delta)^{\frac{\beta}{2}-1}R_{\gamma}^{-1}\Big](-\Delta)^{1-\frac{\beta}{2}} \psi_w}_{\dot{H}^{-\beta/2}} \ d\tau \\
    &\label{Omega estimate} \geq \int_0^t \norm{\Ad_{\gamma^{-1}}^{*}(-\Delta)^{\frac{2s+\beta}{4}}w}_{\dot{H}^{-\beta/2}}^2 \ d\tau \\ 
    \nonumber &\quad - \int_0^t \norm{(-\Delta)^{\frac{2s+\beta}{4}}w}_{\dot{H}^{-\beta/2}}\norm{\symp \Big[(-\Delta)^{\frac{2s+\beta}{4}}, R_{\gamma} (-\Delta)^{\frac{\beta}{2}-1}R_{\gamma}^{-1}\Big](-\Delta)^{1-\frac{\beta}{2}} \psi_w}_{\dot{H}^{-\beta/2}} \ d\tau
\end{align}
To address the first integral, notice that the adjoint $\Ad_\gamma^*$ with respect to the $\dot{H}^{-\beta/2}$ inner product can be expressed in terms of the $L^2$ adjoint $\Ad_\gamma^{0} = D\gamma^\top R_\gamma$ as\footnote{ \ Where $M^\top$ denotes the transpose of the matrix $M$.}
\begin{equation}
    \Ad_\gamma^* = (-\Delta)^{\frac{\beta}{2}} \Ad_{\gamma}^0 (-\Delta)^{-\frac{\beta}{2}}.
\end{equation}
As $\gamma(t)$ is of class $H^{s'}$ we have that $\Ad_{\gamma(t)}^0$ is a bounded operator on $\dot{H}^\frac{\beta}{2}$. Using this we have
\begin{align*}
    \norm{\Big( \Ad_{\gamma(\tau)^{-1}}^{*}(-\Delta)^{\frac{2s+\beta}{4}} \Big)^{-1} v}_{\dot{H}^s} &= \norm{(-\Delta)^{-\frac{2s + \beta}{4}} (-\Delta)^{\frac{\beta}{2}} \Ad_{\gamma(\tau)}^0 (-\Delta)^{-\frac{\beta}{2}} v}_{\dot{H}^s} \\
    &\simeq \norm{\Ad_{\gamma(\tau)}^0 (-\Delta)^{-\frac{\beta}{2}} v}_{\dot{H}^{\frac{\beta}{2}}} \\
    &\lesssim \norm{(-\Delta)^{-\frac{\beta}{2}} v}_{\dot{H}^{ \frac{\beta}{2}}} \\
    &\simeq \norm{v}_{\dot{H}^{-\frac{\beta}{2}}}
\end{align*}
where the suppressed constants above are uniform in time. Hence, for some constant $c_T>0$, we obtain the estimate
\begin{equation*}
    \norm{\Ad_{\gamma(\tau)^{-1}}^{*}(-\Delta)^{\frac{2s+\beta}{4}}w}_{\dot{H}^{-\beta/2}}^2 \geq c_T \norm{w}_{\dot{H}^s}^2
\end{equation*}
for any $w \in T_e\Dex^s$.

Regarding the second integral in \eqref{Omega estimate}, observe that
\begin{align*}
    \Big[(-\Delta)^{\frac{2s+\beta}{4}}, R_{\gamma} (-\Delta)^{\frac{\beta}{2}-1}R_{\gamma}^{-1}\Big] &= R_{\gamma}\Big[R_{\gamma}^{-1}(-\Delta)^{\frac{2s+\beta}{4}}R_{\gamma}, (-\Delta)^{\frac{\beta}{2}-1}\Big]R_{\gamma}^{-1} \\
    &= R_{\gamma}(-\Delta)^{\frac{\beta}{2}-1}\Big[(-\Delta)^{1-\frac{\beta}{2}}, R_{\gamma}^{-1}(-\Delta)^{\frac{2s+\beta}{4}}R_{\gamma}\Big](-\Delta)^{\frac{\beta}{2}-1}R_{\gamma}^{-1}.
\end{align*}
Since the conjugation $R_{\gamma}^{-1}(-\Delta)^{\frac{2s+\beta}{4}}R_{\gamma}$ with a diffeomorphism $\gamma$ is a pseudo differential operator of order $s+\frac{\beta}{2}$, the commutator on the right-hand side above is a pseudo differential operator of order $s-\frac{\beta}{2}+1$, cf. \cite{treves2022analytic}. Hence there exists a constant $C_T>0$ such that
\begin{align}
    \label{Omega estimate 2}\ip{w,\Omega(t)w}_{\dot{H}^s} &\gtrsim c_T\norm{w}_{\dot{H}^s}^2 - C_T \norm{w}_{\dot{H}^s} \bignorm{(-\Delta)^{1-\frac{\beta}{2}} \psi_w}_{\dot{H}^{s+\beta-2}} \\
    \nonumber&\gtrsim c_T \norm{w}_{\dot{H}^s}^2 - C_T \norm{w}_{\dot{H}^s} \norm{w}_{\dot{H}^{s-1}}.
\end{align}

We next turn our attention to the operator $\Gamma(t)$ in \eqref{estimate 1}. From \eqref{lambda and K} and \eqref{lie algebra coadjoint} we have
\begin{align}\label{beta rellich}
    \nonumber \norm{K_0 w}_{\dot{H}^{s+1-\beta}} &= \bignorm{\{(-\Delta)^{1-\frac{\beta}{2}}\psi_{u_0}, \psi_w\}}_{\dot{H}^s} \\
    &\lesssim \bignorm{\nabla (-\Delta)^{1-\frac{\beta}{2}}\psi_{u_0}}_{\mathcal{C}^\rho} \bignorm{\symp \psi_w}_{\dot{H}^s} \\
    \nonumber &\lesssim \norm{w}_{\dot{H}^s}
\end{align}
for any $w \in T_e\Dex^s$ where $\mathcal{C}^\rho$ is the H{\"o}lder-Zygmund space with $\rho>s$, cf. \cite{triebel1983theory}. Hence, for $0\leq \beta < 1$, $K_0$ is a compact operator by the Rellich-Kondrashov lemma. As $\Lambda(\tau)^{-1}$ and $\Phi(\tau)$ are continuous families of bounded linear operators, the integrand in \eqref{omega and gamma} is a compact operator. Hence, $\Gamma(t)$ is itself compact. Combining \eqref{Omega estimate 2} and \eqref{beta rellich}, we arrive at the estimate
\begin{equation}\label{smooth estimate}
    \norm{d\exp_e(tu_0)tw}_{\dot{H}^s} \gtrsim c_T\norm{w}_{\dot{H}^s} - C_T\norm{w}_{\dot{H}^{s-1}} - \norm{\Gamma(t)w}_{\dot{H}^s},
\end{equation}
for $u_0 \in T_e\Dex^{s'}$ and $w \in T_e\Dex^s$.

If we now assume that $u_0 \in T_e\Dex^s$, then for any $\varepsilon>0$ we may approximate $u_0$ by a more regular field $\tilde{u}_0 \in T_e\Dex^{s'}$ such that $\norm{u_0-\tilde{u_0}}_{\dot{H}^s} < \varepsilon$. As the exponential map is smooth in the $\dot{H}^s$ topology, its derivative $u_0 \mapsto d\exp_e(tu_0)$ must depend smoothly on $u_0$ and therefore be locally Lipschitz. Hence, from the triangle inequality and \eqref{smooth estimate} with appropriately modified $\tilde{c}_T$, $\tilde{C}_T$ and $\tilde{\Gamma}(t)$, we have
\begin{align*}
    \norm{d\exp_e(tu_0)tw}_{\dot{H}^s} &\geq \norm{d\exp_e(t\tilde{u}_0)tw}_{\dot{H}^s} - \norm{d\exp_e(tu_0)tw-d\exp_e(t\tilde{u}_0)tw}_{\dot{H}^s} \\
    &\gtrsim \tilde{c}_T\norm{w}_{\dot{H}^s} - \tilde{C}_T\norm{w}_{\dot{H}^{s-1}} - \bignorm{\tilde{\Gamma}(t)w}_{\dot{H}^s} \\
    &\quad - t\norm{d\exp_e(tu_0)-d\exp_e(t\tilde{u}_0)}_{L(\dot{H}^s)} \norm{w}_{\dot{H}^s} \\
    &\geq \tilde{c}_T\norm{w}_{\dot{H}^s} - \tilde{C}_T\norm{w}_{\dot{H}^{s-1}} - \bignorm{\tilde{\Gamma}(t)w}_{\dot{H}^s} - T \mathcal{C}_T \varepsilon \norm{w}_{\dot{H}^s}
\end{align*}
where $\mathcal{C}_T$ denotes a uniform (in time) upper bound on the Lipschitz constant of $u_0 \mapsto d\exp_e(tu_0)$. As $\varepsilon>0$ is arbitrary we then have
\begin{equation}\label{non-smooth estimate}
    \norm{d\exp_e(tu_0)tw}_{\dot{H}^s} \gtrsim \tilde{c}_T\norm{w}_{\dot{H}^s} - \tilde{C}_T\norm{w}_{\dot{H}^{s-1}} - \bignorm{\tilde{\Gamma}(t)w}_{\dot{H}^s}.
\end{equation}

With this estimate in hand, it now follows that $d\exp_e(tu_0)$ has finite-dimensional kernel and closed range in $T_{\gamma(t)}\Dex^s$. The argument is standard. To see why this is true note that, if $d\exp_e(tu_0)w_k$ converges in $\dot{H}^s$ with $w_k$ a bounded sequence in $\dot{H}^s$, then $w_k$ must contain a subsequence which converges in $\dot{H}^{s-1}$ as well as a subsequence for which $\tilde{\Gamma}(t)w_k$ converges in $\dot{H}^s$. We may assume without loss of generality that both subsequences are simply $w_k$. Substituting $w_N - w_M$ into \eqref{non-smooth estimate}, we have
\begin{align*}
    \norm{w_N - w_M}_{\dot{H}^s} &\lesssim \norm{d\exp_e(tu_0)t w_N - d\exp_e(tu_0)t w_M}_{\dot{H}^s} \\
    &\quad + \norm{w_N-w_M}_{\dot{H}^{s-1}} + \bignorm{\tilde{\Gamma}(t)(w_N-w_M)}_{\dot{H}^s}.
\end{align*}
Consequently $w_k$ is Cauchy in $\dot{H}^s$ and hence converges. Now, if the kernel were infinite-dimensional it would contain an orthonormal sequence with a convergent subsequence, which is a contradiction. Hence, $\dim \ker d\exp_e(tu_0)<\infty$.

Next we decompose $T_e\Dex^s$ into a $\dot{H}^s$ orthogonal sum $\ker(d\exp_e(tu_0))\oplus X^s_{u_0}$. There exists a constant $c>0$ such that, for all $w$ in $X^s_{u_0}$, we have $\norm{w}_{\dot{H}^s} \leq c \norm{d\exp_e(tu_0)w}_{\dot{H}^s}$ which implies closed range. If such a constant did not exist then there would be a sequence $w_k$ in $X^s_{u_0}$ such that $\norm{w_k}_{\dot{H}^s}=1$ and $\norm{d\exp_e(tu_0)w_k}_{\dot{H}^s}<\frac{1}{k}$ containing a convergent subsequence. Passing to the limit and using the fact that $X^s_{u_0}$ is closed, this would imply that there exists $w_\infty \in X^s_{u_0}$ with $\norm{w_\infty}_{\dot{H}^s}=1$ and $d\exp_e(tu_0)w_\infty = 0$, a contradiction.

The above shows that, for $u_0 \in T_e\Dex^s$, the operator $d\exp_e(tu_0)$ is semi-Fredholm. Finally, as at $t=0$ $d\exp_e(0)$ is the identity on $T_e\Dex^s$, the theorem follows from the fact that the index function is continuous on the space of semi-Fredholm operators.
\end{proof}

\subsection{Conjugate Points and the Breakdown of Fredholmness}

As mentioned in the Introduction, the distinction between two- and three-dimensional ideal hydrodynamics can be observed through the Fredholm properties of the $L^2$ exponential map. For example, Fredholmness in the 2D case precludes accumulation of conjugate points along finite geodesic segments, while such clustering is known to occur in the 3D setting, cf. \cite{ebin2006singularities, preston2006on}. In this section we amplify Theorem \ref{fredholmness} by establishing, via the Hydrodynamical Morse Index Theorem, cf. \cite{misiolek2010fredholm}, an upper bound on the number of conjugate points along a finite geodesic segment for the $\beta < 1$ metrics; precluding clusters of such points. As the interpolation parameter approaches the critical value $\beta=1$, this upper bound approaches infinity. We also provide an explicit example to illustrate that this bound is attained. This shows that the breakdown of Fredholmness occurs precisely at $\beta=1$; in analogy with 3D ideal hydrodynamics.
\begin{theorem}\label{counting}
    Let $0\leq \beta < 1$. Given any initial velocity $u_0 \in T_e\Dex^s$ let $\gamma(t)$ be the corresponding geodesic of the metric \eqref{beta metric}, which exists for at least some time $T>0$. The number of conjugate points along the segment $\gamma \vert_{[0,T]}$ is bounded above by $\aleph_\beta = \aleph_\beta(u_0, T)<\infty$.
\end{theorem}
As in \cite{misiolek2010fredholm} consider the set of smooth (in time) vector fields along $\gamma\vert_{[0,T]}$ which vanish at the endpoints 
\begin{equation}
\Big\{V: [0,T] \xrightarrow{C^\infty} T\Dex^s \ \vert \ V(t) \in T_{\gamma(t)}\Dex^s \text{ with } V(0)=V(T)=0\Big\}.
\end{equation}
Let $\dot{\mathscr{H}}^1_\beta$ denote the completion of this space under the norm induced by the inner product
$$ \langle V, W\rangle_{\dot{\mathscr{H}}^1_\beta} = \int_0^T \langle \nabla_{\dot{\gamma}} V , \nabla_{\dot{\gamma}} W \rangle_{\dot{H}^{-\beta/2}} \, dt$$

\n where $\nabla_{\dot{\gamma}}$ denotes the covariant derivative operator of the $\dot{H}^{-\beta/2}$ metric in $\Dex^s$. The index form along $\gamma$ is then a symmetric bilinear operator on $\dot{\mathscr{H}}^1_\beta \times \dot{\mathscr{H}}^1_\beta$ given by
\begin{equation}\label{index form}
    I_\beta(V,W) = \int_0^T \Big( \ip{\nabla_{\dot{\gamma}} V , \nabla_{\dot{\gamma}} W}_{\dot{H}^{-\beta/2}} - \ip{\mathcal{R}_\gamma (V) ,W}_{\dot{H}^{-\beta/2}} \Big) \ dt
\end{equation}
where $\mathcal{R}_\gamma(\cdot)=\mathcal{R}(\cdot, \dot{\gamma})\dot{\gamma}$ is the curvature tensor of the $\dot{H}^{-\beta/2}$ metric along $\gamma$ in $\Dex^s$. We recall the following result from \cite{preston2006on}.
\begin{lemma}\label{rewrite index}
    The index form \eqref{index form} can be rewritten as
    \begin{equation}\label{index form rewritten}
        I_\beta(V ,W) = \int_0^T \Big( \ip{\Ad_{\gamma}\partial_t v, \Ad_{\gamma} \partial_t w}_{\dot{H}^{-\beta/2}} + \ip{K_0(v), \partial_t w}_{\dot{H}^{-\beta/2}} \Big) \, dt
    \end{equation}
    where $V = d_eL_\gamma v$, $W = d_eL_\gamma w$ with $v, w \in T_e\Dex^s$ and $K_0$ is given in \eqref{lambda and K}.
\end{lemma}
As in finite-dimensional Riemannian geometry, the Hydrodynamical Morse Index Theorem in \cite{misiolek2010fredholm} states that the number of conjugate points along the geodesic segment $\gamma\big\vert_{[0,T]}$ is equal to the dimension of the space on which this index form is negative definite. Hence, it will suffice to construct a finite-dimensional subspace, $X_\beta$, of $\dot{\mathscr{H}}^1_\beta$ such that the index form $I_\beta$ is non-negative on its $\dot{H}^{-\beta/2}$-orthogonal complement, $Y_\beta$. We proceed to the proof of the main theorem.
\begin{proof}[Proof of Theorem ~\ref{counting}]
    Using Lemma \ref{rewrite index}, for $W = d_eL_\gamma w \in \dot{\mathscr{H}}^1_\beta$ we have
    \begin{align}\label{index estimate 1}
        I_\beta(W, W) &= \int_0^T \Big( \norm{\Ad_\gamma \partial_t w}_{\dot{H}^{-\beta/2}}^2 + \ip{K_0(w), \partial_t w}_{\dot{H}^{-\beta/2}} \Big) \ dt \\
        \nonumber &\geq \int_0^T \Big( \delta \norm{\partial_t w}_{\dot{H}^{-\beta/2}}^2 + \ip{K_0(w), \partial_t w}_{\dot{H}^{-\beta/2}} \Big) \ dt
    \end{align}
    where $\displaystyle \delta = \inf_{t \in [0,T]} \norm{\Ad_\gamma^{-1}}_{L^2}^{-2} > 0$ is independent of $\beta$.
    
    In order to simplify the above expression we let $w(t)=e^{-\frac{t}{2\delta}K_0}v(t)$ for some vector field $v(t)$. Clearly both fields vanish on the same subset of $[0,T]\times M$. Furthermore, since $K_0$ is anti-symmetric in the $\dot{H}^{-\beta/2}$ inner product, we have that $e^{-\frac{t}{2\delta}K_0}$ is an orthogonal operator and hence preserves the $\dot{H}^{-\beta/2}$ norm. Hence from \eqref{index estimate 1} we have
    \begin{align*}
        I_\beta(W, W) &\geq \int_0^T \bigg( \delta \Bignorm{e^{-\frac{t}{2\delta}K_0}\Big(\partial_t v - \frac{1}{2\delta}K_0(v)\Big)}_{\dot{H}^{-\beta/2}}^2 \\
        &\quad + \Bigip{e^{-\frac{t}{2\delta}K_0} K_0(v), e^{-\frac{t}{2\delta}K_0}\Big(\partial_t v - \frac{1}{2\delta}K_0(v)\Big)}_{\dot{H}^{-\beta/2}} \bigg) \ dt \\
        &= \int_0^T \bigg( \delta \bignorm{\partial_t v - \frac{1}{2\delta}K_0(v)}_{\dot{H}^{-\beta/2}}^2 + \bigip{K_0(v), \partial_t v - \frac{1}{2\delta}K_0(v)}_{\dot{H}^{-\beta/2}} \bigg) \ dt \\
        &=\int_0^T \bigg( \delta \norm{\partial_t v}_{\dot{H}^{-\beta/2}}^2 - \frac{1}{4\delta}\norm{K_0(v)}_{\dot{H}^{-\beta/2}}^2 \bigg) \ dt .
    \end{align*}
    Expanding $\displaystyle v(t,x)=\sum_{k=1}^{\infty} \sin\left(\frac{k \pi t}{T}\right)v_k(x)$ where $v_k$ are time-independent fields we then have
    \begin{align*}
        I_\beta(W, W) &\geq \sum_{k=1}^\infty  \int_0^T \bigg( \frac{\delta k^2 \pi^2}{T^2} \cos^2\left(\frac{k \pi t}{T}\right) \norm{v_k}_{\dot{H}^{-\beta/2}}^2 - \frac{1}{4\delta} \sin^2\left(\frac{k \pi t}{T}\right) \norm{K_0(v_k)}_{\dot{H}^{-\beta/2}}^2\bigg) \ dt \\
        &= \frac{T}{2} \sum_{k=1}^\infty \left( \frac{\delta k^2 \pi^2}{T^2} \norm{v_k}_{\dot{H}^{-\beta/2}}^2 - \frac{1}{4\delta} \norm{K_0(v_k)}_{\dot{H}^{-\beta/2}}^2 \right).
    \end{align*}
    Expressing each $v_k = \symp \psi_k$ and using an analogous estimate to \eqref{beta rellich}, namely
    \begin{align}\label{low beta rellich}
    \norm{K_0 w}_{\dot{H}^{-\frac{\beta}{2}}} &\lesssim \bignorm{\nabla (-\Delta)^{1-\frac{\beta}{2}}\psi_{u_0}}_{\mathcal{C}^\rho} \bignorm{\symp \psi_w}_{\dot{H}^{-\frac{\beta}{2}}} \\
    \nonumber &\lesssim \norm{w}_{\dot{H}^{-\frac{\beta}{2}}}
    \end{align}
    for any $w \in T_e\Dex^s$ and $\rho>1+\frac{\beta}{2}$, we have
    \begin{align}\label{index form est}
        I_\beta(W, W) &\geq \frac{T}{2} \sum_{k=1}^\infty \left( \frac{\delta k^2 \pi^2}{T^2} \norm{\psi_k}_{\dot{H}^{1-\beta/2}}^2 - \frac{1}{4\delta} \norm{K_0(\symp \psi_k)}_{\dot{H}^{-\beta/2}}^2 \right) \\
        \nonumber &\geq \frac{T}{2} \sum_{k=1}^\infty \left( \frac{\delta k^2 \pi^2}{T^2} \norm{\psi_k}_{\dot{H}^{1-\beta/2}}^2 - \frac{C}{4\delta} \norm{\psi_k}_{\dot{H}^{\beta/2}}^2 \right),
    \end{align}
    with $C>0$ is a constant (independent of $\beta$) and $u_0 = \symp \psi_0 \in T_e\Dex^s$.
    
    Finally, we expand each $\displaystyle \psi_k(x) = \sum_{n=1}^\infty a_{k,n} \phi_n(x)$ where $a_{k,n}$ are constants and $\phi_n$ are non-constant eigenfunctions of the Laplacian with eigenvalues $0<\lambda_n \nearrow \infty$, normalized so that $\norm{\phi_n}_{L^2}=1$ and for any $\alpha \in \R$ we have
    \begin{align*}
        \norm{\psi_k}_{\dot{H}^\alpha}^2 &= \norm{(-\Delta)^{\frac{\alpha}{2}}\psi_k}_{L^2}^2 = \Bignorm{\sum_{n=1}^\infty a_{k,n} \lambda_n^{\frac{\alpha}{2}} \phi_n}_{L^2}^2 = \sum_{n=1}^\infty a_{k,n}^2 \lambda_n^{\alpha}.
    \end{align*}
    The estimate \eqref{index form est} becomes
    \begin{align*}
        I_\beta(W, W) &\geq \frac{T}{2} \sum_{k=1}^\infty \sum_{n=1}^\infty a_{k,n}^2 \lambda_n^{\frac{\beta}{2}}\left( \frac{\delta k^2 \pi^2}{T^2} \lambda_n^{1-\beta}- \frac{C}{4\delta} \right).
    \end{align*}
    It follows that if
    \begin{equation*}
        \frac{\delta k^2 \pi^2}{T^2} \lambda_n^{1-\beta}- \frac{C}{4\delta} \geq 0
    \end{equation*}
    for each $(k,n)$ such that $a_{k,n}\neq 0$, then we have $I_\beta(W,W)\geq 0$. This condition is equivalent to
    \begin{equation}\label{lambda requirement}
        \lambda_n \geq \left(\frac{CT^2}{4\delta^2k^2\pi^2}\right)^{\frac{1}{1-\beta}}.
    \end{equation}
    Notice that, for fixed $\beta<1$, as $\lambda_n \nearrow \infty$ and since the right hand side of \eqref{lambda requirement} is a decreasing function of $k$, there can be only finitely many $k$, say $1 \leq k \leq k_{\text{max}}$, for which \eqref{lambda requirement} fails to hold for all $n\geq1$. If we consider such $k$, we can approximate the number of $n$ for which \eqref{lambda requirement} does not hold using the Weyl Law \cite{weyl1912das} which states that 
    \begin{equation*}
        N(\lambda)=\#\{n \ \vert \ \lambda_n \leq \lambda\} \approx \frac{\mu(M)}{2\pi}\lambda + O(1).
    \end{equation*}
    Let $A=\big\{(k,n) \ \vert \ \eqref{lambda requirement} \text{ does not hold}\big\}$ and $X_A = \text{span}\big\{\sin\left(\frac{k \pi t}{T}\right)\phi_n\big\}_{(k,n) \in A}$ with closure under the $\dot{H}^{-\beta/2}$ norm denoted by $X_{A,\beta}$. Observe now that
    \begin{align*}
        \aleph_\beta &:= \dim X_{A,\beta} = \# A = \sum_{k=1}^{k_{\text{max}}} N\left(\left(\frac{CT}{4\delta^2k^2\pi^2}\right)^{\frac{1}{1-\beta}}\right) \approx \sum_{k=1}^{k_{\text{max}}} \left(\frac{CT}{4\delta^2k^2\pi^2}\right)^{\frac{1}{1-\beta}} + O(1) < \infty.
    \end{align*}
    Finally, since we have the $\dot{H}^{-\beta/2}$ orthogonal decomposition $\dot{\mathscr{H}}^1_\beta = X_{A,\beta} \oplus Y_{A,\beta}$ with $I_\beta \big\vert_{Y_{A,\beta}} \geq 0$, the number of conjugate points along $\gamma \vert_{[0,T]}$ is bounded above by $\aleph_\beta$.
\end{proof}
An immediate consequence of the above is
\begin{corollary} We have
    \begin{equation*}
        \lim_{\beta \rightarrow 1}\aleph_\beta = \lim_{\beta \rightarrow 1} \sum_{k=1}^{k_{\text{max}}} \left(\frac{CT}{4\delta^2k^2\pi^2}\right)^{\frac{1}{2-2\beta}} + O(1) = \infty
    \end{equation*}
\end{corollary}
We conclude with an example illustrating the number of conjugate points along a finite geodesic segment approaching infinity as $\beta \rightarrow 1$.

\begin{example}
    Following Washabaugh \cite{washabaugh2016the} we consider the case of $M={\mathbb{S}}^2$ with spherical coordinates $(\vartheta,\varphi)$. The rotation $\partial_\vartheta$ is then a steady state solution to \eqref{beta SQG} for any $0 \leq \beta \leq 1$ with corresponding stream function $-\cos(\varphi)$. Note that this is in fact an eigenfunction of the Laplacian with eigenvalue $2$. Letting $w=\symp \psi$, $v_L = \symp \sigma$ and noting that $\gamma(t,\vartheta,\varphi)=(\vartheta+t,\varphi)$, \eqref{left jacobi field} now becomes
\begin{equation}\label{example 1}
        \partial_t \sigma = \psi, \qquad \partial_t (-\Delta)^{1-\frac{\beta}{2}} \psi + 2^{1-\frac{\beta}{2}} \partial_\vartheta \psi = 0
\end{equation}
with initial condition $\sigma(0)=0$.

Taking $\psi(t,\vartheta,\varphi) = \xi(t)\phi_n(\vartheta,\varphi)$ where $\phi_n$ is a spherical harmonic such that $(-\Delta)\phi_n = n(n+1)\phi_n$ and $\partial_\vartheta \phi_n = in \phi_n$ then gives:
\begin{equation}\label{example 2}
    \partial_t \sigma = \psi, \qquad (n(n+1))^{1-\frac{\beta}{2}}\partial_t \xi + i n 2^{1-\frac{\beta}{2}} \xi = 0
\end{equation}
whose solutions, with $\xi(0)=1$, are given by
\begin{equation}
    \xi(t) = e^{-int\left(\frac{2}{n(n+1)}\right)^{1-\frac{\beta}{2}}}
\end{equation}
and
\begin{equation}
    \sigma(t) = \left(\frac{i}{n}\left(\frac{n(n+1)}{2}\right)^{1-\frac{\beta}{2}}\right)\left(e^{-int\left(\frac{2}{n(n+1)}\right)^{1-\frac{\beta}{2}}} - 1\right)\phi_n.
\end{equation}
Notice now that $\sigma$ vanishes at $T_n(\beta) = \frac{2 \pi}{n}\left(\frac{n(n+1)}{2}\right)^{1-\frac{\beta}{2}}$ and hence $\gamma(T_n(\beta),\vartheta,\varphi)$ is conjugate to the identity. In particular if $0\leq \beta <1$ we see these points are spread out, whilst if $\beta=1$, we have $T_n(1) = \pi \sqrt{\frac{2(n+1)}{n}}$; forming a sequence of times, and indeed conjugate points, which cluster at $T=\pi\sqrt{2}$.
\end{example}

There are several natural continuations of this work. One is the study of degeneracy phenomena of the induced geodesic distance. For the $L^2$-metric ($\beta=0$), it has been shown that the geodesic distance is a true distance function~\cite{arnold1966sur,michor2005vanishing}. On the other hand,
for the metric corresponding to the SQG equation ($\beta=1$), it vanishes identically on the whole group~\cite{bauer2020vanishing}. This raises the question: for which $\beta$ does this change in behavior occur? Closely related to this is the notion of the corresponding diameter of the diffeomorphism group. This is well understood for the geodesic distance of the $L^2$-metric~\cite{arnold2021topological, shnirelman1985the, brandenbursky2017the}, but the situation for general $\beta$ has not yet been studied.

\bibliographystyle{amsplain}
\bibliography{bibliography}

\vfill

\end{document}